\documentclass{amsart}
\usepackage{eepic,epic}
\usepackage{mathrsfs}
\usepackage{graphics,psfrag}
\usepackage{graphicx}
\usepackage{latexsym}
\usepackage{amssymb}
\usepackage{amsmath}
\usepackage[english]{babel}
\usepackage[mathscr]{eucal}
\usepackage[all]{xy}
\usepackage{epsfig}
\usepackage{enumerate}

\RequirePackage{color}
\definecolor{myred}{rgb}{0.75,0,0}
\definecolor{mygreen}{rgb}{0,0.5,0}
\definecolor{myblue}{rgb}{0,0,0.65}

\RequirePackage{ifpdf}
\ifpdf
  \IfFileExists{pdfsync.sty}{\RequirePackage{pdfsync}}{}
  \RequirePackage[pdftex,
   colorlinks = true,
   urlcolor = myred, 
   citecolor = mygreen, 
   linkcolor = myred, 
   pagebackref,
   pdfpagemode=None,
   bookmarksopen=true]{hyperref}
\else
  \RequirePackage[hypertex]{hyperref}
\fi

\RequirePackage{ae, aecompl, aeguill} 

\newcommand\T{\rule{0pt}{3.1ex}}

\theoremstyle{plain}
\newtheorem{theorem}{Theorem}[section]
\newtheorem{proposition}[theorem]{Proposition}
\newtheorem{lemma}[theorem]{Lemma}
\newtheorem{corollary}[theorem]{Corollary}
\newtheorem{conjecture}[theorem]{Conjecture}

\theoremstyle{definition}

\theoremstyle{remark}
\newtheorem{remark}[theorem]{Remark}

\def\ov{\overline}

    \def\CM{{\mathbb{C}}}
    
    \def\EM{{\mathbb{E}}}
    \def\FM{{\mathbb{F}}}
  \def\gG{{\mathfrak g}}

    \def\KM{{\mathbb{K}}}

    \def\NM{{\mathbb{N}}}
    \def\OM{{\mathbb{O}}}
    
    \def\QM{{\mathbb{Q}}}

    \def\ZM{{\mathbb{Z}}}

    \def\GC{{\mathcal{G}}}
    \def\HC{{\mathcal{H}}}
    
    \def\JC{{\mathcal{J}}}
    \def\KC{{\mathcal{K}}}

    \def\OC{{\mathcal{O}}}

    \def\UC{{\mathcal{U}}}

\def\a{\alpha}
\def\b{\beta}

\def\D{\Delta}
\def\e{\varepsilon}

\def\l{\lambda}

\DeclareMathOperator{\Perv}{{\mathrm{Perv}}}

\DeclareMathOperator{\supp}{{\mathrm{supp}}}

\newcommand{\elem}[1]{\stackrel{#1}{\longto}}

\def\to{\rightarrow}

\def\longto{\longrightarrow}

\def\rtordu{\rightsquigarrow}

\def\lexp#1#2{\kern\scriptspace\vphantom{#2}^{#1}\kern-\scriptspace#2}

\mathchardef\inferieur="321E
\mathchardef\superieur="321F

\def\eqna{\begin{eqnarray*}}
\def\endeqna{\end{eqnarray*}}

\def\ie{{\textit{i}.\textit{e}.}}

\def\p{{}^p\!}
\def\pp{{}^{p_+}\!}
\def\0{{}^0\!}

\def\Sing{{\mathrm{Sing}}}

\def\Spin{{\mathrm{Spin}}}

\def\courte{{\mathrm{sh}}}

\author{Daniel Juteau}

\title[Modular representations and affine Grassmannians]
{Modular representations of reductive groups and geometry of
  affine Grassmannians}

\def\gr{{\GC r}}
\def\eq{\Xi}

\begin{document}


\maketitle

\begin{abstract}
By the geometric Satake isomorphism of Mirkovi\'c and Vilonen,
decomposition numbers for reductive groups can be interpreted as
decomposition numbers for equivariant perverse sheaves on the complex affine
Grassmannian of the Langlands dual group. Using a description of the
minimal degenerations of the affine Grassmannian obtained by Malkin,
Ostrik and Vybornov, we are able to recover geometrically some
decomposition numbers for reductive groups. In the other direction, we
can use some decomposition numbers for reductive groups to prove
geometric results, such as a new proof of non-smoothness results,
and a proof that some singularities are not equivalent (a conjecture
of Malkin, Ostrik and Vybornov). We also give counterexamples to a
conjecture of Mirkovi\'c and Vilonen stating that the stalks of standard
perverse sheaves over the integers on the affine Grassmannian are
torsion-free, and propose a modified conjecture, excluding bad primes.
\end{abstract}

\section*{Introduction}
\label{sec:intro}

In \cite{decperv}, we introduced decomposition numbers for perverse
sheaves and calculated them for simple and minimal singularities. This
has applications in the modular representation theory of Weyl groups,
using a modular Springer correspondence \cite{these}.

In this article, we will use another bridge between representation
theory and geometry, provided by the geometric Satake isomorphism of
Mirkovi\'c and Vilonen \cite{MV}. It is an equivalence of tensor categories
between the category of representations of a reductive algebraic group $G$
over an arbitrary commutative ring $\EM$, and a category of
equivariant perverse sheaves on the complex affine Grassmannian of the
Langlands dual group, with $\EM$ coefficients.

In the first part of this article, we give applications from
geometry to the representation theory of reductive algebraic group
schemes, using results of Malkin, Ostrik and Vybornov describing the
minimal degenerations in the affine Grassmannian \cite{MOV}. In
Section \ref{sec:satake}, we review the results of \cite{MV}. We
deduce that decomposition numbers for reductive groups (that is, the
multiplicities of simple modules in Weyl modules) are decomposition
numbers for perverse sheaves. In Section \ref{sec:min deg}, we review
the results of \cite{MOV}.  It turns out that most minimal
degenerations in the affine Grassmannian are either simple or minimal
singularities, which we have studied in \cite{decperv}. So we can
recover decomposition numbers for reductive groups in this way. First,
we remark in Section \ref{sec:levi} that a Levi lemma known in
representation theory \cite[\S 5.21 (2) p. 230]{JAN} can be
interpreted geometrically by a Levi lemma used by Malkin, Ostrik and
Vybornov. Then, in Section \ref{sec:simple}, for the case of a simple
singularity we recover the fact that for two dominant weights which
are adjacent in the order of all dominant weights, the decomposition
number is one if there is a wall between them, and zero otherwise
\cite[Cor. 6.24 p.249]{JAN}. In Section \ref{sec:min}, for the case of a minimal
singularity, we recover the multiplicity of the trivial module in the
Weyl module of highest weight the dominant short root, which was
computed in \cite[Theorem 1.1]{CPScohI}.
Finally, in Section \ref{sec:torsion}, we remark that the calculations we
did in \cite{cohmin} for minimal singularities, together with the
results in \cite{decperv}, provide counterexamples to a conjecture in
\cite{MV} stating that the standard $\ZM_\ell$-perverse sheaves on the
affine Grassmannian should have torsion-free stalks. We propose to
modify the conjecture by requiring that $\ell$ be good.

In the second part of this article, we go in the other
direction and use some decomposition numbers for reductive
groups to prove geometric results.
In Section \ref{sec:spin}, we compute the decomposition number for
$G = \Spin_{2n + 1}$ corresponding for the weights
$\l = \varpi_1 + \varpi_n$ and $\mu = \varpi_n$ (in the numbering of
\cite{BOUR456}), corresponding to the quasi-minimal singularity of type $ac_n$.
In Section \ref{sec:non-equiv}, we prove non-equivalences of
singularities conjectured in \cite{MOV}. The argument is as follows:
if the singularities were equivalent, then the stalks of the
intersection cohomology complexes would be the same, both in
characteristic zero and in characteristic $\ell$, so the decomposition
numbers for perverse sheaves would be the same, and thus also the
decomposition numbers for reductive groups. We then see that there is
always a prime number $\ell$ for which these decomposition numbers are
different.
Finally, in Section \ref{sec:non-smooth}, we give a
representation-theoretic proof of the fact that the smooth locus of an
orbit on the affine Grassmannian is reduced to that orbit (this result
was already proved, in different ways, in \cite{EM} and \cite{MOV}).
For each minimal degeneration, we give a prime number $\ell$ for which
the corresponding decomposition number is non-trivial, which implies
that it is not $\FM_\ell$-smooth, and thus not smooth.

The methods and results in this work show again that it is useful to
consider perverse sheaves with coefficients of arbitrary
characteristic. Considering the intersection cohomology complexes
modulo all possible primes $\ell$, we get a finer invariant than just
the characteristic zero case. In this way, one can prove non-trivial
geometric results such as non-smoothness or non-equivalences, which
cannot be seen in the rational intersection cohomology stalks.
In the other direction, a deep understanding of the geometry of the
affine Grassmannians, including the determination of the stalks of the
intersection cohomology complexes in positive characteristic, would give
the modular characters of the reductive groups.

\subsection*{Acknowledgements}

The results in this article were obtained at the Mathematical Sciences
Research Institute in Berkeley, during the programs ``Representation
theory of finite groups and related topics'' and ``Combinatorial
representation theory''. The author is very grateful to the organizers
and to several people who taught him a lot about representations of
reductive groups, particularly Pierre Baumann, Cédric Lecouvey,
Zungzhou Lin and Leonard Scott. The lectures by Jonathan Brundan were
also very helpful. 

\setcounter{tocdepth}{1}
\tableofcontents

\part{From geometry to representation theory}
\label{part:geom to rep}

\section{Geometric Satake isomorphism}
\label{sec:satake}

By the geometric Satake isomorphism of Mirkovi\'c and Vilonen
\cite{MV}, the category of representations of a connected reductive
algebraic group scheme over any commutative ring $\EM$ is naturally
equivalent to a certain category of equivariant perverse sheaves on
the complex affine Grassmannian of the Langlands dual group $G^\vee$,
with $\EM$ coefficients.

Let $G$ be a split connected reductive algebraic group scheme defined over $\ZM$.
We fix a split maximal torus $T$ in $G$.
Then we have the character lattice $X(T)$ and
the cocharacter lattice $Y(T)$, with the canonical perfect pairing
$\langle-,-\rangle : X(T) \times Y(T) \to \ZM$, and the root systems $\Phi$ in $X(T)$
and $\Phi^\vee$ in $Y(T)$.
We also fix a Borel subgroup $B$ in $G$ containing $T$, corresponding
to some choice of subset of positive roots $\Phi^+$ in $\Phi$, defined by
a basis $\D$. We denote by $X(T)^+$ the set of dominant weights.

Then the irreducible $G_\CM$-modules are the Weyl modules $V_\CM(\l)$,
for $\l \in X(T)^+$. Let $\ell$ be a prime number.  There is a
procedure of reduction modulo $\ell$.  The choice of a highest weight
vector $v_+$ in $V_\CM(\l)$ (which is unique up to a non-zero scalar)
determines an integral Weyl module $V_\ZM(\l)$ (it is a module for
$\UC_\ZM$, Kostant's $\ZM$-form of the enveloping algebra $\UC$ of the
Lie algebra $\gG$ of $G$).  Let $\FM = \ov\FM_\ell$. Then the Weyl
module $V(\l)$ for $G_\FM$ can be defined as $\FM \otimes_\ZM
V_\ZM(\l)$. It turns out that this is the universal highest weight
module of highest weight $\l$ for $G_\FM$.  It has a unique simple
quotient $L(\l)$. The $L(\l)$, for $\l \in X(T)^+$, form a full set of
representatives of isomorphism classes of simple modules for
$G_\FM$. We also have the induced modules $H^0(\l)$, for $\l \in
X(T)^+$. There is a natural morphism $V(\l) \to H^0(\l)$, and $L(\l)$
can also be seen as the image of this morphism.  We denote by
$G_\FM$-mod the category of rational representations of $G_\FM$. 

To use a uniform notation, if $\EM$ is any commutative ring, we will
denote by $V_\EM(\l)$ and $H^0_\EM(\l)$ the Weyl
and induced modules for $G_\EM$, defined similarly, and if $\EM$ is a
field, we will write $L_\EM(\l)$ for the simple modules. We want to let
$\EM$ vary in an $\ell$-modular system $(\KM,\OM,\FM)$, both for the
representation theory and for the coefficients of the perverse
sheaves. That is, we can take for $\OM$ a discrete valuation ring with
quotient field $\KM$ and of characteristic zero and residue field
$\FM$ of characteristic $\ell$, for example finite extensions of
$\QM_\ell$, $\ZM_\ell$ and $\FM_\ell$, as in \cite{decperv}. But here, since
we work we the complex topology, we can take arbitrary commutative
rings as coefficients, such as $\CM$, $\ZM$ and $\ov\FM_\ell$.

The character of the Weyl module $V(\l)$ is the same as its
counterpart over $\CM$. It is given by Weyl's character formula.
The characters of the simple modules $L(\mu)$ are not known in
general. Their determination is equivalent to the determination of the
multiplicities $d^G_{\l\mu} := [V(\l):L(\mu)]$. Obviously, we have
$d^G_{\l\l} = 1$, and $d^G_{\l\mu}\neq 0$ implies $\mu \leq \l$
(for the usual order on $X(T)$, that is, $\l - \mu$ has to be a
non-negative linear combination of positive roots).

Now let $G^\vee_\CM$ denote the Langlands dual group, defined over the
complex numbers. Let $\KC = \CM((t))$ and $\OC = \CM[[t]]$.  We denote
by $\gr$ the affine Grassmannian $G^\vee(\KC)/G^\vee(\OC)$.  It is an
ind-variety (a direct limit of varieties).  If $\EM$ is any commutative
ring, let $\Perv_{G^\vee_\OC}(\gr,\EM)$ denote the category of
$G^\vee_\OC$-equivariant perverse sheaves on $\gr$ with $\EM$
coefficients. The $G^\vee_\OC$-orbits on $\gr$ are parametrized by
$X(T)^+$. We denote the orbit corresponding to $\l$ by $\gr_\l$. It is
of dimension $d_\l := 2 \langle \l, \rho^\vee \rangle$, where
$\rho^\vee$ is the sum of the fundamental coweights.  For $\l$ in
$X(T)^+$, we set $\p\JC_!(\l,\EM) = \p j_{\l !}
(\ov\gr_\l,\EM[d_\l])$, where $j_\l$ is the inclusion of $\gr_\l$ in
$\gr$ and $p$ is the middle perversity, and similarly for
$\p\JC_{!*}(\l,\EM)$ and $\p\JC_*(\l,\EM)$.  If $\EM = \OM$, we can
consider these three extensions relative to the dual perversity $p_+$
instead of $p$ \cite{BBD,decperv}.

Mirkovic and Vilonen \cite{MV} related the representation theory
of $G_\EM$ with $\EM$-perverse sheaves on the affine Grassmannian of
$G^\vee$. More precisely, their main result is as follows:

\begin{theorem}[Mirkovic-Vilonen]
We have an equivalence of tensor categories
\[
\eq : G_\EM\text{-mod} \elem{\sim} \Perv_{G^\vee_\OC}(\gr,\EM)
\]
which sends the natural morphism $V_\EM(\l) \to H^0_\EM(\l)$ to the natural
morphism $\JC_!(\l,\EM) \to \JC_*(\l,\EM)$. Thus, if $\EM$ is a field,
$\eq$ sends the simple $L_\EM(\l)$ to the simple $\JC_{!*}(\l,\EM)$.

Moreover, this equivalence is compatible with the extension of scalars
$\KM \otimes_\OM -$ and the modular reduction $\FM \otimes_\OM -$ (if
we take a torsion-free object in either category for $\EM = \OM$, its
modular reduction is an object of the corresponding category for
$\EM = \FM$).
\end{theorem}

They also prove that $\p \JC_!(\l,\OM) \simeq \p \JC_{!*}(\l,\OM)$.
This implies
\[
\p \JC_!(\l,\OM) \simeq \pp\JC_!(\l,\OM) \simeq \p \JC_{!*}(\l,\OM)
\]
and 
\[
\pp \JC_{!*}(\l,\OM) \simeq \p\JC_*(\l,\OM) \simeq \pp
\JC_*(\l,\OM)
\]
and also
\[
\FM \p \JC_!(\l,\OM) \simeq \p \JC_!(\l,\FM)
\simeq \FM \pp \JC_!(\l,\OM) \simeq \FM \p \JC_{!*}(\l,\OM)
\]
and
\[
\FM \pp \JC_{!*}(\l,\OM) 
\simeq \FM \p \JC_*(\l,\OM) 
\simeq \p \JC_*(\l,\FM)
\simeq \FM \pp \JC_*(\l,\OM)
\]
Here, we denote simply by $\FM(-)$ the functor of modular reduction.
See \cite{decperv} for general results of this kind.

We can consider the decomposition numbers for $G_\OC^\vee$-equivariant
perverse sheaves on $\gr$, defined by
\[
d^\gr_{\l\mu} := [\FM \p\JC_{!*}(\l,\OM) : \p\JC_{!*}(\mu,\FM)]
= [\p\JC_!(\l,\FM) : \p\JC_{!*}(\mu,\FM)]
\]
Let us give an immediate consequence of the result of Mirkovi\'{c} and
Vilonen in terms of decomposition numbers:

\begin{corollary}
For $\l$, $\mu$ in $X(T)^+$, we have
\[
d^G_{\l\mu} = d^\gr_{\l\mu}
\]
\end{corollary}

Therefore, this decomposition number will just be denoted by $d_{\l\mu}$.

\section{Minimal degenerations}
\label{sec:min deg}

In \cite{decperv}, we introduced decomposition numbers for perverse
sheaves, described some of their properties, and computed them for
simple surface singularities and for minimal singularities.
This had applications for the modular representation theory of Weyl
groups, using a modular Springer correspondence making a link with
modular perverse sheaves on the nilpotent cone: we showed in
\cite{these} that decomposition numbers for Weyl groups are
particular cases of decomposition numbers for perverse sheaves on the
nilpotent cone.

It turns out that simple and minimal singularities also occur in the
affine Grassmannian, and in fact most minimal degenerations in $\gr$
are of either kind, by the work of Malkin, Ostrik and
Vybornov \cite{MOV}. In non-simply-laced types, a few others show up,
which they call quasi-minimal.

A minimal degeneration in a stratified space is a pair of strata which
are adjacent in the order given by the inclusion of closures. In the
affine Grassmannian $\gr$, these are parametrized by pairs of adjacent
dominant weights $(\l,\mu)$ (in the usual order), \ie\ such that $\l >
\mu$ and there is no dominant weight $\nu$ with $\l > \nu > \mu$.
Such a pair is also called a minimal degeneration, and is denoted by
$\l \rtordu \mu$. They were classified by Stembridge \cite{STEM}.

For $\b$ in $Q$, we denote by $\supp(\b)$ the Dynkin subdiagram
consisting in the simple roots appearing in the decomposition of
$\b$ with a non-zero coefficient. For a minimal degeneration
$\l \rtordu \mu$, $\supp(\l - \mu)$ is clearly connected.
For $I\subset \D$ and  $\l = \sum_{\a\in\D} \l_\a \varpi_\a$ a weight, we set
$\l_I = \sum_{\a \in I} \l_\a \varpi_\a$.

\begin{theorem}[Stembridge]
\label{th:stembridge}
Let $\l > \mu$ be a pair of weights. We set $\b = \l - \mu$ and
$I = \supp(\b) \subset \D$.
Then we have $\l \rtordu \mu$ if and only if one of the following holds:
\begin{enumerate}[{Case} (1)]
\item
\label{case:simple}
$\b$ is a simple root.

\item
\label{case:short}
$\b$ is the short dominant root of $\Phi_I$ and
$\langle \mu, \a^\vee \rangle = 0$ for all $\a$ in $I$.

\item
\label{case:ac}
$\b$ is the short dominant root of $\Phi_I$, $\Phi_I$ is of type
$B_n$, and $\mu_I = \varpi_\a$, where $\a$ is the unique short simple
root in $I$.

\item
\label{case:ag}
$\Phi_I$ is of type $G_2$, $\l_I = \varpi_{\a_1} + \varpi_{\a_2}$,
$\mu_I = 2 \varpi_{\a_1}$, where $I = \{\a_1,\a_2\}$ with $\a_1$ short and
$\a_2$ long.

\item
\label{case:cg}
$\Phi_I$ is of type $G_2$, $\l_I = \varpi_{\a_2}$, $\mu_I = \varpi_{\a_1}$,
where $I = \{\a_1,\a_2\}$ as in the previous case.
\end{enumerate}
\end{theorem}

Note that, in the two last cases, we have
$\b = \varpi_{\a_2} - \varpi_{\a_1} = \a_1 + \a_2$.
In this situation, if we assume $\Phi$ to be irreducible, then $\Phi = \Phi_I$.

Now we come to the descrpition of the minimal degenerations. For the
notion of smooth equivalence of singularities and the notation
$\Sing$, we refer to \cite{KP1}.

\begin{theorem}[Malkin-Ostrik-Vybornov]
\label{th:mov}
Let $\l \rtordu \mu$ be a minimal degeneration.
Let $\b = \l - \mu$ and $I = \supp(\b)$.
We set $\l = \sum_{\a\in\D} \l_\a \varpi_\a$ and
$\mu = \sum_{\a\in\D} \mu_\a \varpi_\a$.
\begin{enumerate}[{Case} (1)]
\item
If $\b$ is a simple root, then
$\Sing(\ov\gr_\l, \gr_\mu)$ is a Kleinian singularity of type
$A_{\l_\b - 1} = A_{\mu_\b + 1}$.

\item
If $\b$ is the short dominant root of $\Phi_I$ and $\mu_I = 0$, then
$\Sing(\ov\gr_\l, \gr_\mu)$ is a minimal singularity of the type of
$\Phi_I^\vee$, the root subsystem of $\Phi^\vee$ generated by the
$\a^\vee$, $\a\in I$.
\end{enumerate}
\end{theorem}

So most minimal degenerations are Kleinian or minimal singularities,
and we will be able to apply our study of decomposition numbers for perverse
sheaves on these singularities in \cite{decperv}. Malkin, Ostrik and
Vybornov called the remaining minimal degenerations
\emph{quasi-minimal singularities}. They are denoted by
$ac_n$ in case \eqref{case:ac},
$ag_2$ in case \eqref{case:ag},
and $cg_2$ in case \eqref{case:cg}.

The minimal degenerations can be studied using a transverse slice
$S(\l,\mu) = \ov\gr_\l \cap L^{< 0}G.\mu$, with the notation of
\cite{MOV}. We denote by $d(\l,\mu)$ the
dimension of $S(\l,\mu)$. We set:
\[
m_\mu(\l,q) = \sum_{i\geq 0} \dim \HC^{i - d(\l,\mu)}_\mu \ \p\JC_{!*}(\ov\gr_\l,\KM).q^i
\]

\begin{proposition}
The codimension $d = d(\l,\mu)$ and intersection cohomology invariants $m_\mu(\l,q)$
over $\KM$ of the minimal degenerations are given by:
\begin{enumerate}[{Case} (1)]
\item
$d = 2$ and $m_\mu(\l,q) = 1$;

\item
$d = 2h^\vee(G^\vee) - 2$, where $h^\vee(G^\vee)$ is the dual Coxeter
  number of $G^\vee$,
and $m_\mu(\l,q) = \sum_{i = 1}^{t} q^{e_i - 1}$, where $t$ is the
number of long simple roots in $\Phi_I$, and the $e_i$ are the
exponents of $W_I$;

\item
$d = 2n$ and $m_\mu(\l,q) = \sum_{i = 0}^{n - 1} q^i$, as for the
  minimal singularity $a_n$;

\item
$d = 4$ and $m_\mu(\l,q) = 1 + q$, as for the minimal singularity $a_2$;

\item
$d = 4$ and $m_\mu(\l,q) = 1$, as for the minimal singularity $c_2$.
\end{enumerate}
\end{proposition}

Note that, for simple and minimal singularities, we computed the
local intersection cohomology over the integers in \cite{decperv}. For
minimal singularities, this uses the computation of the integral
cohomology of the minimal orbit in \cite{cohmin}.

Malkin, Ostrik and Vybornov were able to prove the following result in
a way completely different from Evens and Mirkovi\'c \cite{EM}:

\begin{theorem}
\label{th:smooth locus}
The smooth locus of $\ov\gr_\l$ is just $\gr_\l$.
\end{theorem}

They argue as follows.
It is enough to check that $\ov\gr_\l$
is singular along every irreducible component of the boundary
$\ov\gr_\l - \gr_\l$, which are precisely the Schubert varieties
$\ov\gr_\mu$ for all minimal degenerations $\l\rtordu \mu$. So they have
to check that all minimal degenerations are singular, and they can use
their classification. Simple singularities and minimal singularities
are known to be singular; for $ac_n$ and $ag_2$ singularities, they
computed the rational intersection cohomology and found that they are
not rationally smooth, hence they are not smooth; finally, for the
$cg_2$ singularity, which is rationally smooth, they use Kumar's
criterion. The equivariant multiplicity is the integer $27$ divided by
a product of weights, which indicates that it is rationally smooth, but
not smooth (otherwise, the numerator would be $1$). We will provide a
representation-theoretic proof of the theorem in Section \ref{sec:non-smooth}.

They also conjecture that the singularities $a_2$, $ac_2$ and $ag_2$
(resp. $c_2$ and $cg_2$) are pairwise non-equivalent. We will prove
this conjecture in Section \ref{sec:non-equiv}.

\section{A Levi lemma}
\label{sec:levi}

Malkin, Ostrik and Vybornov use the following geometric Levi lemma
\cite[\S 3]{MOV}:
\begin{lemma}
Let $I \subset \D$. If $\l - \mu \in \NM I$, then we have (with
obvious notations):
\[
\Sing(\ov\gr_\l, \gr_\mu) = \Sing(\ov\gr^{L_I^\vee}_{\l_I}, \gr^{L_I^\vee}_{\mu_I})
\]
\end{lemma}

If $\l$ and $\mu$ are as in the lemma, then any pair of weights $\nu
\geq \zeta$ in the interval $[\mu,\l]$ also satisfies this property,
so the lemma can be applied for all of them. Thus, we have the same
intersection cohomology stalks in both situations (for $G$ or $L_I$),
either with ordinary or modular coefficients, for all the interval.
Since we only consider constant local systems, this implies that the
decomposition numbers for perverse sheaves are the same for all the
interval. By the geometric Satake isomorphism, the same is true on the
representation theoretic side. So we recover a Levi lemma which was
already known in representation theory
\cite[\S 5.21 (2) p. 230]{JAN}:
\begin{corollary}
If $I\subset \D$ and $\l - \mu \in \ZM I$, then we have (with obvious notations):
\[
[V(\l) : L(\mu)] = [V_I(\l) : L_I(\mu)]
\]
\end{corollary}

\section{Simple singularities}
\label{sec:simple}

\begin{theorem}
In case \eqref{case:simple}, we have
$d_{\l\mu} = 1$ if $\ell$ divides $\l_\b$, and $0$ otherwise.
\end{theorem}

\begin{proof}
The decomposition number for a simple singularity is given in
\cite[\S 4.3]{decperv}.
\end{proof}

Thus we recover geometrically a result that can be found in
\cite[Cor. 6.24 p.249]{JAN}.

\section{Minimal singularities}
\label{sec:min}

\begin{theorem}
In case \eqref{case:short}, we have
$d_{\l\mu} = \dim_{\FM_\ell} \FM_\ell \otimes_\ZM
P(\Phi_{I_\courte})/Q(\Phi_{I_\courte})$,
where $I_\courte$ is the set of roots of minimal length in $I$, for any
$W_I$-invariant scalar product.
\end{theorem}

\begin{proof}
We have a minimal singularity of the type of $\Phi_I^\vee$.
By \cite{cohmin} and \cite[\S 5]{decperv}, the decomposition number modulo $\ell$ is
given by this formula: the long simple coroots correspond to the short
simple roots.
\end{proof}

Thus we recover a result of Cline, Parshall and Scott
\cite[Theorem 1.1]{CPScohI}. Their formulation is slightly different
(they use the rank of the Cartan matrix of $\Phi_{I_\courte}$), but it
is easily seen to be equivalent (the Cartan matrix of a root system is
the matrix of the inclusion of the coroot lattice into the coweight
lattice, for some choice of basis of simple coroots, determining the
basis of fundamental coweights; moreover, the finite abelian groups $P/Q$ and
$P^\vee/Q^\vee$ are in duality).

\section{On the torsion of the stalks}
\label{sec:torsion}

Mirkovi\'{c} and Vilonen conjectured \cite[Conjecture 13.3]{MV} that the
stalks of the standard sheaves
$\p\JC_!(\l,\ZM_\ell) = \p\JC_{!*}(\l,\ZM_\ell)$
 are torsion-free (they actually state this conjecture for
$\ZM$ coefficients, but it is equivalent to this property holding for
$\ZM_\ell$ coefficients for all prime numbers $\ell$).

But we explained in \cite{decperv} that the stalk at $0$ of the
intersection cohomology complex of a minimal singularity is given by
the first half of the cohomology of the minimal nilpotent orbit, and
by the calculations in \cite[\S 3]{cohmin}, we see that
there is $\ell$-torsion for $\ell = 2$ in types $B_n$, $C_n$, $D_n$, $F_4$
for $\ell = 2$, $3$ in types $E_6$, $E_7$, for $\ell = 2$, $3$, $5$
in type $E_8$, and for $\ell = 3$ in type $G_2$. However, we note that
for a minimal singularity, all the primes that appear are bad primes
for $G^\vee$ (recall that we consider the affine Grassmannian for
$G^\vee$). Therefore, we may propose the following modified conjecture:

\begin{conjecture}
If $\ell$ is good for $G^\vee$, then the stalks of the standard
perverse sheaf $\JC_!(\l,\ZM_\ell) = \JC_{!*}(\l,\ZM_\ell)$ are torsion-free.
\end{conjecture}

\part{From representation theory to geometry}
\label{part:rep to geom}

\section{A decomposition number for $G = \Spin_{2n + 1}$}
\label{sec:spin}

In \cite{MOV}, it is conjectured that the quasi-minimal singularities
are not equivalent to minimal singularities. We we will be
able to prove these non-equivalences, using decomposition numbers for
reductive groups. In order to deal with the $ac_n$ singularity, we
will determine the corresponding decomposition number for $G =
\Spin_{2n + 1}$, the simply-connected simple group of type $B_n$ (we
have not found it in the literature). For the proof that the
singularities $a_n$ and $ac_n$ are not equivalent, we actually only
need a much weaker statement, which gives a necessary condition for
the decomposition number to be non-zero, and which follows from the
strong linkage principle (see remark \ref{rem:linkage}). But for the
representation theoretic proof of non-smoothness, we need a
non-vanishing result. Besides, along the way we will calculate a
bilinear form which is likely to be interpreted geometrically.
Let us mention that the tables in \cite{LUB} were very useful to
conjecture the result of this section.

We will use the notation of Bourbaki for the root system $\Phi$ of
type $B_n$. So we view $\Phi$ in $\ZM^n = \bigoplus_{i = 1}^n \ZM
\e_i$, and the simple roots are $\a_i = \e_i - \e_{i + 1}$, for $1
\leq i \leq n - 1$, and $\a_n = \e_n$. We denote by $W$ the Weyl
group, and by $(\cdot|\cdot )$ the $W$-invariant perfect pairing for which
$(\e_1,\ldots,\e_n)$ is orthonormal. We have
$\rho = \varpi_1 + \cdots \varpi_n
= \frac{1}{2}((2n - 1)\e_1 + (2n - 3)\e_2 + \cdots + \e_n)$.

The decomposition number we are interested is $d_{\l\mu}$, where
$\l = \varpi_1 + \varpi_n = \frac{1}{2}(3\e_1 + \e_2 + \cdots + \e_n)$,
and $\mu = \varpi_n = \frac{1}{2}(\e_1 + \cdots + \e_n)$.
We note that $\l - \mu = \varpi_1 = \e_1 = \a_1 + \cdots + \a_n$,
and $\l \rtordu \mu$ is a minimal degeneration.

\begin{lemma}
\label{lemma:orbits}
We have $|W.\l| = n.2^n$ and $|W.\mu| = 2^n$.
\end{lemma}

\begin{proof}
Since $\l$ is dominant, we have $W_\l = W_{I_\l}$, where
$I_\l = \{\a_2,\ldots,\a_{n-1}\}$ is the set of simple roots orthogonal
to $\l$. It is a Weyl group of type $A_{n - 2}$. So we have 
\[
|W.\l| = |W:W_\l| = \frac{2^n n!}{(n-1)!} = n.2^n.
\]
Similarly, we have $W_\mu = W_{I_\mu}$, with
$I_\mu = \{\a_1,\ldots,\a_{n-1}\}$.
It is a Weyl group of type $A_{n - 1}$. So we have
\[
|W.\mu| = |W : W_\mu| = \frac{2^n n!}{n!} = 2^n.
\]
\end{proof}

\begin{lemma}
\label{lemma:weyl char}
The characters of the Weyl modules $V(\l)$ and $V(\mu)$ are given by
\[
\chi(\l) = \sum_{w\in W/W_\l} e(w.\l)\ +\ n \sum_{w \in W/W_\mu} e(w.\mu)
\]
and
\[
\chi(\mu) = \sum_{w \in W/W_\mu} e(w.\mu).
\]
\end{lemma}

\begin{proof}
The result for $\chi(\mu)$ is clear since $\mu$ is a minuscule weight.

The only dominant weight below $\l$ (in the usual order) is $\mu$.
By $W$-invariance, the multiplicity of the weights in the orbit of
$\l$ is one, and so we only have to determine the multiplicity of the
weights in the orbit of $\mu$.
For this, we can use Freudenthal's formula
\cite[Chap. 8, \S 9, ex. 5]{BOUR456}, which in our case is more
convenient that Weyl's formula.

\[
\begin{array}{l}
\left( (\l + \rho \mid \l + \rho) - (\mu + \rho \mid \mu + \rho) \right)
\dim V(\l)_\mu
\\
\T
\quad = 2 \sum_{\a\in\Phi^+} \sum_{m \geq 1} (\mu + m\a \mid \a)
\dim V(\l)_{\mu + m\a}
\\
\T
\quad = 2 \sum_{1\leq i\leq j\leq n} (\mu + \a_i + \cdots + \a_j \mid \a_i +
\cdots + \a_j)
\\
\T
\quad = 2 \sum_{1\leq i\leq j\leq n - 1}
\left( \frac{1}{2}(\e_1 + \cdots \e_n) + \e_i - \e_{j + 1} \mid \e_i -
\e_{j + 1}\right)
\\
\T
\qquad +\ 2 \sum_{1\leq i \leq n} \left( \frac{1}{2}(\e_1 + \cdots \e_n) +
\e_i \mid \e_i\right)
\\
\T
\quad = 2 . \frac{n(n-1)}{2} . 2 + 2.n.\frac{3}{2}
\\
\T
\quad = n(2n - 2 + 3) = n(2n + 1)
\end{array}
\]

In the calculation, we have used the fact that
$\l - \mu = \a_1 + \cdots + \a_n$, so that the only positive roots
$\a$ which can contribute are of the form $\a_i + \cdots + \a_j$, and
they do so only for the first multiple ($m = 1$).

Now 
\[
\begin{array}{l}
\left( (\l + \rho \mid \l + \rho) - (\mu + \rho \mid \mu + \rho) \right)
\\
\T
\quad = (\l + \mu + 2\rho \mid \l - \mu)
\\
\T
\quad = ((2n + 1)\e_1 + (2n - 2)\e_2 + (2n - 4)\e_3 + \cdots + 2\e_n \mid\e_1)
\\
\T
\quad = 2n + 1.
\end{array}
\]

Thus $\dim V(\l)_\mu = n$, and we are done.
\end{proof}

In the construction of the Weyl module $V(\l)$, we have chosen a
highest weight vector $v$. We will now give an explicit basis for
$V(\l)_\mu$, in terms of the Chevalley generators $f_i$, $1\leq i \leq n$.

\begin{proposition}
\label{prop:weyl basis}
A basis for $V(\l)_\mu$ is given by $(v_1,\ldots,v_n)$, where
\[
v_i = f_i f_{i+1} \ldots f_n f_{i-1} \ldots f_1  v
\]
\end{proposition}

\begin{proof}
The weight space $V(\l)_\mu$ is certainly generated by elements of the
form $f_{i_1} \ldots f_{i_n} v$, where $(i_1,\ldots,i_n)$ is some
permutation of $(1,\ldots,n)$. Since there are no multiplicities, we
do not have to worry about divided powers.

Now, in order to get a non-zero vector, we have to take $i_n$ equal to
$1$ or $n$. If we choose $i_n = 1$, then we get
$f_1 v$, a vector of weight $\varpi_2 + \varpi_n$, and for the next
step we have to choose between $2$ and $n$ for $i_{n-1}$ to get a
non-zero result, and so on. Using the commutation relations, we can
assume that we first apply $f_1$, $f_2$, \ldots, $f_{i - 1}$, and then
$f_n$, $f_{n - 1}$, \ldots, $f_i$, for some $i$ between $1$ and $n$.

Thus $V(\l)_\mu$ is generated by $(v_1,\ldots,v_n)$. By Lemma
\ref{lemma:weyl char}, we have $\dim V(\l)_\mu = n$, so this is a basis.
By the way, we see that it is also a basis of $V_\ZM(\l)_\mu$.
\end{proof}

Alternatively, we can use the tableaux combinatorics, as explained in
\cite{KN}. We can see $V(\l) = V(\varpi_1 + \varpi_n)$ as the
submodule of $V(\varpi_1) \otimes V(\varpi_n)$ generated by a highest
weight vector. The Weyl module $V(\varpi_1)$ is the natural
representation of dimension $2n + 1$, and $V(\varpi_n)$ is the spin
representation, of dimension $2^n$. With the notation of
\cite[\S 5]{KN}, we have
\[
v_i =
{\tiny
\begin{array}{|c|c|}
\hline
1 & i + 1\\
\hline
2\\
\cline {1-1}
\vdots\\
\cline {1-1}
\T \widehat{i + 1}\\
\cline {1-1}
\vdots\\
\cline {1-1}
n\\
\cline {1-1}
\T\ov{i + 1}
\\
\cline {1-1}
\end{array}
}
\]
for $1 \leq i \leq n - 1$, and
\[
v_n =
{\tiny
\begin{array}{|c|c|}
\hline
1 & 0\\
\hline
2\\
\cline {1-1}
\vdots
\\
\cline {1-1}
n
\\
\cline {1-1}
\end{array}
}
\]

\begin{proposition}
The matrix of the contravariant form $(\cdot | \cdot)$ on
$V_\ZM(\l)_\mu$, in the basis $(v_1,\ldots,v_n)$, is given by:
\[
\begin{pmatrix}
2 & 1 & 0 & \cdots & 0\\
1 & \ddots & \ddots & \ddots & \vdots\\
0 & \ddots & 2 & 1 & 0\\
\vdots & \ddots & 1 & 2 & 1\\
0 & \cdots & 0 & 1 & 3
\end{pmatrix}
\]
The elementary divisors of this matrix are
$(1, 1, \ldots, 1, 2n + 1)$. 
\end{proposition}

\begin{proof}
Using the commutation relations and the fact that $v$ is a highest
weight vector, we find:
\[
\begin{array}{rcl}
(v_i \mid v_i)
&=& (v \mid e_1 \ldots e_{i - 1} e_n \ldots e_i
            f_i \ldots f_n f_{i - 1} \ldots f_1 v)
\\
\\ &=&
\prod_{j = i}^n
\langle \l - \a_1 - \cdots \a_{i - 1} - \a_{j + 1} - \cdots - \a_n,
\a_j^\vee \rangle
\\
\\ && \times \quad
\prod_{j = 1}^{i - 1}
\langle \l - \a_1 - \cdots \a_{j - 1}, \a_j^\vee \rangle
\\
\\
&=&
\begin{cases}
2 \text{ if } 1\leq i \leq n - 1,\\
3 \text{ if } i = n.
\end{cases}
\end{array}
\]

Similarly, for $1 \leq i < j \leq n$, we have
\[
\begin{array}{rcl}
(v_i \mid v_j)
&=& (v \mid e_1 \ldots e_{i - 1} e_n \ldots e_i
            f_j \ldots f_n f_{j - 1} \ldots f_1 v)
\\
\\ &=&
\prod_{k = i + 1}^{j - 1}
\langle \l - \a_1 - \cdots \a_{i - 1},
\a_k^\vee \rangle
\\
\\ && \times \quad
\prod_{k = j}^n
\langle \l - \a_1 - \cdots \a_{i - 1} - \a_{k + 1} - \cdots - \a_n,
\a_k^\vee \rangle
\\
\\ && \times \quad
\prod_{k = 1}^i
\langle \l - \a_1 - \cdots \a_{k - 1}, \a_k^\vee \rangle
\\
\\
&=& 0^{j - i - 1} \times 1^{n + 1 - j} \times 1^i
\\
\\
&=&
\begin{cases}
1 \text{ if } j = i + 1,\\
0 \text{ otherwise.}
\end{cases}
\end{array}
\]

This determines the matrix of the contravariant form (which is
symmetric). Now, this matrix has the same elementary divisors as:
\[
\begin{pmatrix}
1&2&1&0&\cdots&0\\
0&1&2&\ddots&\ddots&\vdots\\
\vdots&\ddots&\ddots&\ddots&\ddots&0\\
\vdots&&\ddots&1&2&1\\
0&\cdots&\cdots&0&1&3\\
2&1&0&\cdots&0&0
\end{pmatrix}
\]
By induction, we can replace the last line by
$(0,\ldots,0,i+1,i,0,\ldots,0)$ without changing the elementary
divisors (where the first non-zero entry is in column $i$), for
$i$ up to $n - 1$. Then we can replace the last line by
$(0,\ldots,0,2n + 1)$, hence the result.
\end{proof}

\begin{theorem}
\label{th:ac_n}
Let $G = \Spin_{2n+1}$, the simply-connected simple group of type
$B_n$, and let $\l = \varpi_1 + \varpi_n$, $\mu = \varpi_n$ in the
numbering of \cite{BOUR456}.  Then we have $d_{\l\mu} = 1$ if $\ell$
divides $2n + 1$, and $0$ otherwise.
\end{theorem}

\begin{proof}
This follows from the preceding results.
\end{proof}

\section{Non-equivalences of singularities}
\label{sec:non-equiv}

\begin{theorem}
\begin{enumerate}
\item The singularities $a_n$ and $ac_n$ are not equivalent.

\item The singularities $a_2$, $ac_2$ and $ag_2$ are pairwise
  non-equivalent.

\item The singularities $c_2$ and $cg_2$ are not equivalent.
\end{enumerate}
\end{theorem}

\begin{proof}
For each pair of singularities, we proceed as follows: if the two
singularities were equivalent, then
they would have the same intersection cohomology stalks, both in
characteristic zero and in characteristic $\ell$, and thus the
corresponding decomposition numbers for perverse sheaves should be the
same; but these are also decomposition numbers for a reductive group, and
in each case, we see that the decomposition numbers differ for some
primes $\ell$.

The decomposition number for the singularity $a_n$ is $1$ if $\ell$
divides $n + 1$, and $0$ otherwise, whereas the decomposition number
for the singularity $ac_n$ is $1$ if $\ell$ divides $2n + 1$, and $0$
otherwise. Moreover $n + 1$ and $2n + 1$ are coprime.
Hence these two singularities are not equivalent.

The decomposition number for $a_2$ (resp. $ac_2$, $ag_2$) is $1$ for
$\ell = 3$ (resp. $\ell = 5$, $7$), and $0$ otherwise
(for the $ac_2$ and $ag_2$ cases, one can for example consult the tables in
\cite{LUB}). Hence these singularities are pairwise non-equivalent.

The decomposition number for $c_2$ and $cg_2$ is $1$ if $\ell = 2$ 
(resp. $\ell = 3$) and $0$ otherwise (again, one can use the tables in
\cite{LUB} for $cg_2$). Hence these two singularities are not equivalent.
\end{proof}

\begin{remark}
\label{rem:linkage}
We do not need the full strength of Theorem \ref{th:ac_n}, if we just
want to prove the non-equivalence of singularities for the $ac_n$ case. If
$\l \rtordu \mu$ is a minimal degeneration, by the strong linkage
principle,  if $d_{\l\mu} \neq 0$ then we must have
\[
\mu = s_{\b, m\ell} (\l)
 = \l -(\langle \l + \rho, \b^\vee \rangle - m\ell)\b
\]
where $\b = \l - \mu$, hence $\ell$ must divide
$\langle \l + \rho, \b^\vee \rangle - 1$.
In the case of the singularity $ac_n$, $\b$ is the short dominant
root, and we have
\[
\begin{array}{l}
\langle \l + \rho, \b^\vee \rangle - 1
\\
\quad = \langle 2\varpi_1 + \varpi_2 + \cdots + \varpi_{n - 1} + 2\varpi_n,
2 \a_1^\vee + 2\a_2^\vee + \cdots + 2 \a_{n - 1}^\vee + \a_n^\vee
\rangle - 1
\\
\quad = 4 + 2(n - 2) + 2 - 1 = 2n + 1.
\end{array}
\]

Besides, $n + 1$ and $2n + 1$ are coprime. Hence, choosing a prime number
$\ell$ dividing $n + 1$, the decomposition number is one for
$a_n$, and zero for $ac_n$, so the two singularities cannot be
equivalent.

In the rank two case, choosing $\ell = 7$, the
decomposition number is $1$ for $ag_2$ and $0$ for $ac_2$ (in this
case $2n + 1 = 5$), which proves that they are also non-equivalent.
\end{remark}

\section{Non-smoothness}
\label{sec:non-smooth}

We will now give a new proof of the fact that the smooth locus of the
closure of an orbit in the affine Grassmannian is reduced to the orbit itself.

\begin{proof}[Representation-theoretic proof of Theorem \ref{th:smooth locus}]
As in the proof of Malkin-Ostrik-Vybornov, we need only check that
each minimal degeneration $\l \rtordu \mu$ is non-smooth. To this end,
we will prove that there is some prime number $\ell$ for which it is
not $\FM_\ell$-smooth. For this, it is enough to prove that there is
some prime number $\ell$ for which the decomposition number
$d_{\l\mu}$ for the corresponding perverse sheaves is non-trivial. By
the geometric Satake isomorphism, this is a decomposition number for
the reductive group $G$.

In case \ref{case:simple}, we have $d_{\l\mu} = 1$ if $\ell$ divides
$n + 1$ and $0$ otherwise, thus the $A_n$ singularity is
not $\FM_\ell$-smooth for $\ell$ dividing $n + 1$.

In case \ref{case:short}, we have 
$d_{\l\mu} = \dim_{\FM_\ell} \FM_\ell \otimes_\ZM
P(\Phi_{I_\courte})/Q(\Phi_{I_\courte})$, thus
this minimal singularity is not $\FM_\ell$-smooth for
$\ell$ dividing $|P(\Phi_{I_\courte})/Q(\Phi_{I_\courte})|$.

In case \ref{case:ac}, we have $d_{\l\mu} = 1$ if $\ell$ divides
$2n + 1$ and $0$ otherwise, thus the $ac_n$ singularity is
not $\FM_\ell$-smooth for $\ell$ dividing $2n + 1$.

In case \ref{case:ag}, we have $d_{\l\mu} = 1$ if $\ell = 7$ 
and $0$ otherwise, thus the $ag_2$ singularity is
not $\FM_7$-smooth.

In case \ref{case:cg}, we have $d_{\l\mu} = 1$ if $\ell = 3$ 
and $0$ otherwise, thus the $cg_2$ singularity is
not $\FM_3$-smooth.

So, in all cases, there is at least one prime number $\ell$ for which
the decomposition number is non-trivial, so all the minimal
degenerations are non-smooth, and the result follows.
\end{proof}

\def\cprime{$'$}

\end{document}